%% file: main.tex
\documentclass[11pt,a4paper]{article}
\usepackage[T1]{fontenc}
\usepackage{lmodern,amsmath,amssymb,amsthm,mathtools}
\usepackage[margin=25mm]{geometry}
\usepackage{microtype,booktabs,array,enumitem,xurl,tikz}
\usepackage[hidelinks]{hyperref}
\setlist{itemsep=3pt,topsep=5pt}
\newtheorem{theorem}{Theorem}
\newtheorem{lemma}[theorem]{Lemma}
\newtheorem{proposition}[theorem]{Proposition}
\newtheorem{corollary}[theorem]{Corollary}
\theoremstyle{remark}
\newcommand{\Fp}{\mathbb F_p}
\newcommand{\ind}{\mathbf1}
\newcommand{\aug}{\operatorname{aug}}
\newcommand{\hats}{\widehat s}
\newcommand{\hatP}{\widehat P}
\title{Failure of direct finiteness for group algebras\\in every odd characteristic}
\author{Elena I. Bunina\quad Winston Lee\\[3pt]
Alexei Miasnikov\quad Kyrylo Muliarchyk}
\date{October 15, 2026}
\hypersetup{pdftitle={Failure of direct finiteness for group algebras in every odd characteristic},pdfauthor={Elena I. Bunina, Winston Lee, Alexei Miasnikov, Kyrylo Muliarchyk}}
\begin{document}
\maketitle
\begin{abstract}
For every odd prime $p$, we construct a finitely presented group $G_p$ and elements $a,b,c\in\mathbb F_p[G_p]$ such that $ab=1$, $ac=0$, and $c\ne0$. Thus $\mathbb F_p[G_p]$, and the group algebra over every field of characteristic $p$, fails to be directly finite. The construction combines a single exceptional self-product pair with signed bipartite cancellation. New integer matching data realize the exceptional self-intersection while retaining a strict squared-word decay estimate. We verify the adaptations needed for the earlier bounded-path and planar-localization arguments. A minimal disk picture protects the identity coefficient of $c$, so the proof does not require torsion-freeness.
\end{abstract}

\section{Introduction}
A unital ring $R$ is \emph{directly finite} if $ab=1$ implies $ba=1$ for all $a,b\in R$. It is \emph{stably finite} if every finite square matrix ring over $R$ is directly finite. For a group $G$ and a field $k$, the group algebra $k[G]$ consists of finite sums of group elements with coefficients in $k$, with multiplication extended from the group law. Kaplansky's direct-finiteness question asks whether $k[G]$ is directly finite for every $k$ and $G$. The question includes groups with torsion.

Kaplansky formulated the positive-characteristic question in \emph{Fields and Rings}, first published in 1969. We cite his original discussion in the second edition~\cite[pp.~122--123]{kaplansky}: after stating direct finiteness and its matrix version in characteristic zero, he asks whether one-sided inverses in group algebras are also two-sided in characteristic $p$. The characteristic-zero argument uses finite von Neumann algebras. The positive-characteristic question is the remaining case of what is now called Kaplansky's direct-finiteness conjecture.

When $G$ is finite, $k[G]$ and all its finite square matrix rings are finite-dimensional over $k$, so their direct finiteness follows from linear algebra, even when the characteristic divides $|G|$. For infinite groups, finite support alone does not give a finite-dimensional ambient algebra. Elek and Szab\'o proved stable finiteness for group algebras of sofic groups over arbitrary division rings~\cite[Corollary~4.7]{elekszabo}. Their theorem covers, in particular, amenable and residually finite groups. Consequently a group admitting a counterexample over even one field must be nonsofic.

Direct finiteness is distinct from the zero-divisor and unit questions for torsion-free group algebras. Nonzero zero divisors alone do not provide a one-sided inverse, and torsion-freeness is not part of the direct-finiteness conjecture. Our construction uses three finite sums, arranged so that the two product identities are automatic while one coefficient is protected from collapse.

\subsection{The result}
For each odd prime $p$, we construct a group $G_p$ and elements
\begin{equation}\label{eq:certificate}
ab=1,\qquad ac=0,\qquad c\ne0\qquad\text{in }\Fp[G_p].
\end{equation}
These identities force $ba\ne1$, since $ba=1$ would give $c=(ba)c=b(ac)=0$.

\begin{theorem}\label{thm:main}
For every odd prime $p$, there is a finitely presented group $G_p$ and elements $a,b,c\in\Fp[G_p]$ satisfying \eqref{eq:certificate}. Consequently $\Fp[G_p]$ is not directly finite. The same witnesses work over every field of characteristic $p$.
\end{theorem}

The last assertion follows from the natural injective map $\Fp[G_p]\hookrightarrow k[G_p]$ for every field $k$ of characteristic $p$. Thus both the coefficient field and the quantifier over primes are explicit: the witnesses already have coefficients in the prime field, and a group is obtained for each odd prime. The group and the finite sums may depend on $p$. Their construction selects suitable finite labelled graphs for sufficiently large multiplicity; it does not provide a computed small presentation.

\begin{corollary}\label{cor:nonsofic}
For every odd prime $p$, the finitely presented group $G_p$ constructed in Theorem~\ref{thm:main} is nonsofic.
\end{corollary}
\begin{proof}
If $G_p$ were sofic, the theorem of Elek and Szab\'o~\cite[Corollary~4.7]{elekszabo} would make $\Fp[G_p]$ directly finite, contradicting Theorem~\ref{thm:main}.
\end{proof}

\subsection{Relation to earlier constructions}
The September 23 manuscript~\cite{sep} gives a characteristic-two zero-divisor construction for a finitely presented torsion-free group. It introduces the bounded-path and deterministic planar-localization arguments used here. Failure of direct finiteness over $\mathbb F_2$ was established in~\cite{oct}. In odd characteristic, the September 26 construction~\cite{oddone} gives a counterexample for one specified odd prime over a field of order $p^4$, using a different construction with torsion. That result does not assert failure for every odd prime.

Ivanov~\cite{ivanov} extends the signed zero-divisor construction to every odd prime. His framework uses a bipartite graph on one side, projective-plane types with additional labels, and a squared-word estimate. These are also ingredients of the present proof. The zero-divisor witnesses cannot be used unchanged for direct finiteness: their signed left factor has augmentation zero, whereas a right-invertible element must have nonzero augmentation. Section~2 makes this obstruction explicit.

The contribution here is an exceptional-pair design compatible with the odd-characteristic signed construction. It gives failure of direct finiteness for every odd prime, with witnesses over the prime field. The new type frequencies and integer corrections simultaneously enforce the algebraic intersection conditions, exact inverse-label matching, and a strict word-decay estimate. The noncollapse proof then protects only the coefficient needed for the direct-finiteness certificate.

\subsection{The construction and its safeguards}
The defining data are two finite graphs immersed into a rose whose oriented edges are labelled by a signed alphabet. We impose the labels of all closed graph paths as relations, using a finite set of fundamental loops. A vertex in either chosen root component then determines a group element by reading a path from the root. The first graph is bipartite, and its two vertex colors supply the signs in the left factor.

The outgoing-label sets have intersection cardinality $1$ modulo $p$ for every relevant pair of vertices, except for the first root paired with itself, where the cardinality is $0$ modulo $p$. Pairing each edge label with its inverse cancels the corresponding signed products. The single exceptional pair leaves the identity coefficient in $ab$, while all terms in $ac$ cancel. These calculations hold for every matching of the prescribed slots and allow different vertices to represent the same group element.

Two further steps are essential. First, the type multiplicities must be realizable by actual matchings: inverse labels need exactly equal slot counts, and the first graph must respect its bipartition. A congruence modulo $p$ does not suffice for these integer requirements. Second, the resulting quotient must retain $c\ne0$. We ensure that the second root is the only vertex in its component whose group label is the identity, yielding $[c]_1=1$. No separation of the other vertex labels is needed.

The probabilistic argument supplies this root separation. A positive test function gives strict contraction of the squared turn-weight matrix. Pool normalization and a small perturbation accommodate unequal matching pools and turns available only at the exceptional root. The bounded-path estimate excludes systems with almost all occurrences paired over distinct graph edges. Deterministic planar localization reduces a hypothetical trivial rooted path to such a bounded system. The required tolerance is chosen before the bounds on its number of paths, lengths, and comparisons.

\subsection{Proof dependencies and organization}
Section~2 proves the algebraic certificate. Sections~3--4 prescribe the types, integer multiplicities, and inverse-label involution. Section~5 establishes strict squared-word decay. Proposition~\ref{prop:bounded-path} in Section~6 states the bounded-path exclusion with its quantifiers; its proof gives the normalized stage expectations and the adaptations of~\cite[Proposition~3.1]{sep}. Lemma~\ref{lem:planar-localization} in Section~7 states the deterministic extraction used from~\cite[Proposition~4.5]{sep}, after which a minimal disk picture proves root noncollapse.

The common-block decomposition and planar-separator argument are taken from the cited September manuscript. We specify the hypotheses used and verify their preservation under the present changes; the proof does not require an asphericity or torsion-freeness conclusion. The bibliography links the three OpenAI manuscripts at a fixed repository revision so that these dependencies can be checked against the versions used here.

\section{The algebra: one controlled exception to cancellation}
\subsection{Labels, paths, and the defining group}
Let $T$ be a finite alphabet with a fixed-point-free involution $t\mapsto\bar t$. Choose one generator from each pair $\{t,\bar t\}$ and interpret the other letter as its inverse. These pairs label the two orientations of edges of a rose $F$. The involution is part of the construction; it need not come from a symmetry of the projective plane used below.

A finite graph $\Gamma\to F$ is \emph{immersed} if each vertex has at most one outgoing edge of each label. Write $S_x\subseteq T$ for the outgoing labels at $x$. An immersed path has no immediate edge reversal. It is \emph{cyclically immersed} if it is closed and also has no reversal where its last and first edges meet.

Take finite loopless immersed graphs $\Gamma_A,\Gamma_B$, with vertex sets $A,B$. In each component choose a basepoint and a spanning tree, and let $P_u$ be its tree path to $u$. Define $G$ by imposing, for each non-tree edge $e:u\to v$ labelled $t$, the relator
\[
\operatorname{lab}(P_u)\,t\,\operatorname{lab}(P_v)^{-1}=1.
\]
These fundamental loops form a free basis of the component's fundamental group. Thus this is a finite presentation and every closed path in either graph has trivial label in $G$; a loop at another vertex is conjugated to the component's basepoint.

Choose roots $r\in A$ and $s\in B$, and let $A',B'$ be the vertex sets of their components. The label of a path from $r$ to $x\in A'$ gives a well-defined element $g_x\in G$, because the labels of closed graph paths are trivial. Define $h_y$ similarly for $y\in B'$. Thus
\[
g_r=h_s=1,\qquad g_{x\cdot t}=g_xt,\qquad h_{y\cdot t}=h_yt,
\]
where $x\cdot t$ is the endpoint of the outgoing $t$-edge. These are group labels of vertices, whereas the members of $T$ are edge labels.

Assume $\Gamma_A$ is bipartite. Put $\epsilon(x)=1$ on its positive part and $\epsilon(x)=-1$ on its negative part, with $\epsilon(r)=1$. Set
\begin{equation}\label{eq:witnesses}
a=\sum_{x\in A'}\epsilon(x)g_x,\qquad
b=\sum_{x\in A'}g_x^{-1},\qquad
c=\sum_{y\in B'}h_y^{-1}.
\end{equation}
The graph $\Gamma_B$ need not be bipartite. These sums count vertices with multiplicity even when their group labels coincide.

\subsection{Why the unchanged zero-divisor factor has no right inverse}
For every finite bipartite graph, signed incidence counting gives the integer identity
\begin{equation}\label{eq:incidence}
\sum_x\epsilon(x)\deg(x)=0:
\end{equation}
the endpoints of each edge contribute opposite signs. If every degree is $1$ modulo $p$, then $\sum_x\epsilon(x)=0$ in $\Fp$, separately on every component. Hence the left factor in the unchanged signed zero-divisor construction has augmentation zero. It cannot have a right inverse, since $ab=1$ implies $\aug(a)\aug(b)=1$.

This obstructs that particular witness. The repair changes one degree and, more generally, one intersection count.

\subsection{The intersection conditions}
We require
\begin{align}
|S_x\cap S_{x'}|&\equiv1-\ind_{(x,x')=(r,r)}\pmod p
&& (x,x'\in A),\label{eq:AA}\\
|S_x\cap S_y|&\equiv1\pmod p
&& (x\in A,\ y\in B).\label{eq:AB}
\end{align}
Every intersection is $1$ modulo $p$, except the root's self-intersection, which is $0$. On the diagonal, the root therefore has degree $0$ modulo $p$ and every other vertex has degree $1$. Equation \eqref{eq:incidence}, applied to the root component, gives
\[
0\equiv\sum_{x\in A'}\epsilon(x)-1,\qquad\aug(a)=1.
\]

\begin{lemma}[Algebraic certificate]\label{lem:algebra}
Conditions \eqref{eq:AA}--\eqref{eq:AB} imply $ab=1$ and $ac=0$, without any assumption that distinct vertices have distinct group labels.
\end{lemma}
\begin{proof}
For $t\in T$, define
\[
a_t=\sum_{\substack{x\in A'\\t\in S_x}}\epsilon(x)g_x,\qquad
b_t=\sum_{\substack{x\in A'\\t\in S_x}}g_x^{-1},\qquad
c_t=\sum_{\substack{y\in B'\\t\in S_y}}h_y^{-1}.
\]
Following a $t$-edge bijects the $t$-slots with the $\bar t$-slots in the same component. Its endpoints have opposite $A$-signs, and their group labels differ by right multiplication by $t$. Hence
\[
a_{\bar t}=-a_t t,\qquad b_{\bar t}=t^{-1}b_t,\qquad c_{\bar t}=t^{-1}c_t.
\]
Thus $a_{\bar t}b_{\bar t}=-a_tb_t$ and $a_{\bar t}c_{\bar t}=-a_tc_t$. Summing over inverse pairs and then using \eqref{eq:AA}--\eqref{eq:AB} gives, in $\Fp[G]$,
\begin{align*}
0=\sum_ta_tb_t
 &=\sum_{x,x'\in A'}|S_x\cap S_{x'}|\epsilon(x)g_xg_{x'}^{-1}=ab-1,\\
0=\sum_ta_tc_t
 &=\sum_{\substack{x\in A'\\y\in B'}}|S_x\cap S_y|\epsilon(x)g_xh_y^{-1}=ac.
\end{align*}
Only the pair $(r,r)$ is exceptional, and its contribution is $\epsilon(r)g_rg_r^{-1}=1$. Neither connectedness of the defining graphs nor separation of vertex labels enters the calculation.
\end{proof}

\subsection{Exactly how much noncollapse is needed}
For each $g\in G$, the coefficient of $g^{-1}$ in $c$ is
\begin{equation}\label{eq:fibre}
[c]_{g^{-1}}=\#\{y\in B':h_y=g\}\pmod p.
\end{equation}
Thus $c\ne0$ precisely when some fibre of $y\mapsto h_y$ has size nonzero modulo $p$. In particular, it suffices to prove
\begin{equation}\label{eq:root}
h_y=1\quad\Longrightarrow\quad y=s.
\end{equation}
Then $[c]_1=1$. We prove \eqref{eq:root} in Section~\ref{sec:noncollapse}. No separation of the $A$-labels is needed, and even the $B$-labels away from $1$ may coincide.

\section{Choosing types with exactly one exceptional pair}
Fix an odd prime $p$, choose a power
\[
q=p^e\ge100p^4,\qquad r_0=q/p,
\]
and keep $p,q$ fixed. Only a later multiplicity parameter $H$ tends to infinity. The numerical bound is convenient, not optimal.

In $\operatorname{PG}(2,q)$ there are $q^2+q+1$ points and the same number of lines. Each line has $q+1$ points, and distinct lines meet in one point. Choose a line $L$, distinct points $o_1,o_2\in L$, and $z\notin L$. The point $z$ organizes matching corrections; it is not a graph root. Use all projective points as ordinary labels, and add $p$ extra labels $E=\{e_0,\ldots,e_{p-1}\}$. The alphabet has even size, since $q^2+q+1$ and $p$ are both odd.

For every line $\ell\ne L$, prescribe the normal type
\begin{equation}\label{eq:types}
S_\ell=\begin{cases}\ell\cup E,&\ell\text{ passes through }o_1\text{ or }o_2,\\
\ell,&\text{otherwise},\end{cases}
\qquad S_*=(L\setminus\{o_1,o_2\})\cup\{e_0\}.
\end{equation}
Give $r$ type $S_*$ and every other vertex a normal type. Normal types intersect in $1,q+1,p+1$, or $q+p+1$ labels, all congruent to $1$ modulo $p$, including self-intersections. If $\ell\cap L\notin\{o_1,o_2\}$, then $S_*\cap S_\ell$ is that one ordinary point. For an anchor line, its removed ordinary intersection is replaced by $e_0$, again giving one label. Finally $|S_*|=q\equiv0\pmod p$. This proves \eqref{eq:AA}--\eqref{eq:AB}.

Omitting $L$ from the normal types is essential. All degrees lie between $q$ and $q+p+1$. We still have to balance the inverse-label slots exactly and choose frequencies that make long repeated words rare.

\section{Realizing the types by exact integer matchings}
\subsection{The balance equations}
An outgoing occurrence of a label $t$ is a \emph{slot}. Pairing it with a $\bar t$-slot makes an edge. In $A$, paired slots must belong to opposite colors. Let $N_t^\pm$ count the $t$-slots in $A^\pm$, and put
\[
T_t=N_t^++N_t^-,\qquad D_t=N_t^+-N_t^-.
\]
Both oriented matching requirements hold exactly when
\begin{equation}\label{eq:balance}
T_t=T_{\bar t},\qquad D_t=-D_{\bar t}.
\end{equation}
Indeed $N_t^\pm=(T_t\pm D_t)/2$, so these equations say $N_t^+=N_{\bar t}^-$ and $N_t^-=N_{\bar t}^+$. This is the precise equality of incoming and outgoing counts needed for the matchings. The root adds one to both $T_t$ and $D_t$ at every $t\in S_*$.

\subsection{Baseline counts and pencil corrections}
For each $P\in L$, choose a positive even baseline $B_P$. Exactly $q$ normal lines meet $L$ at $P$. Start with $B_P/2$ copies of each such line type in each $A$-color and in $B$. Correct only the type on $zP$, using integers $d_P,u_P$:
\begin{equation}\label{eq:multiplicity}
m_\pm(\ell)=\frac{B_P+\ind_{\ell=zP}(u_P\pm d_P)}2,
\qquad P=\ell\cap L.
\end{equation}
Add the root of type $S_*$ to $A^+$. In $B$, retain $B_P/2$ copies of each normal type. We require $u_P\equiv d_P\pmod2$ and arrange
\[
\sum_{P\in L}d_P=\sum_{P\in L}u_P=0.
\]
A point of $L$ lies on one corrected line; an affine point $x\ne z$ lies on the corrected line $zx$; and $z$ lies on every corrected line, so the corrections at $z$ cancel.

\subsection{Signed counts and special inverse pairs}
Since $r_0=q/p$ is odd, assign the following values to the $q+1$ points of $L$:
\begin{equation}\label{eq:profiles}
\begin{array}{c|c}
d_P&\text{multiplicity}\\\hline
0&p-1\\
+(r_0-1),\;-(r_0-1)&p-1\text{ each}\\
+j,\;-j\quad(0<j<r_0-1,\ j\text{ even})&p-2\text{ each}\\
+j,\;-j\quad(0<j<r_0-1,\ j\text{ odd})&2\text{ each}
\end{array}
\end{equation}
Writing $r_0=2k+1$, the total is $3(p-1)+2(k-1)(p-2)+4k=q+1$. Symmetry gives $\sum d_P=0$. Set $u_P=0$ for even $d_P$; for the two copies of each odd signed value, assign $u=1$ to one and $u=-1$ to the other. This gives the required parity and $\sum u_P=0$.

For a sufficiently large integer $H$, put
\begin{equation}\label{eq:parameters}
\alpha=\frac{p+3}{2p+4},\qquad\gamma=\frac{p+3}{(p+2)^2},\qquad K=4(p+2)H.
\end{equation}
Assign the following data to distinct points of $L$:
\begin{equation}\label{eq:special}
\begin{array}{c|rrc|c}
P&d_P&u_P&B_P/K&\text{inverse partner}\\\hline
o_1&1-r_0&0&\alpha&v_1\\
v_1&r_0-2&-1&\alpha&o_1\\
o_2&-1&1&\alpha&v_2\\
v_2&0&0&\alpha&o_2\\
Z_0&r_0-2&1&2\alpha&e_0\\
Z_i\ (1\le i<p)&r_0-1&0&2\alpha&e_i\\
w&-1&-1&1&z
\end{array}
\end{equation}
Every entry is available in \eqref{eq:profiles}. All remaining points of $L$ have $B_P=K$.

The large signed counts near $r_0$ compensate for the root. Each extra occurs on the two anchor pencils and has normal signed count $d_{o_1}+d_{o_2}=-r_0$. Across $p$ extras this contributes $-q$, balancing the root degree $q$, while the ordinary normal labels contribute $(q+1)\sum_Pd_P=0$ to the total signed degree.

\subsection{Why these baselines work}
Leading slot balance requires equal baselines at $o_i,v_i$, baseline $B_{o_1}+B_{o_2}$ at each $Z_i$, and $\sum_PB_P=qK$ to pair $w$ with $z$. The last condition becomes
\[
(q-p-3)+4\alpha+2p\alpha=q,
\]
which gives the chosen $\alpha$. There are four baselines $\alpha K$, $p$ baselines $2\alpha K$, and $q-p-3$ baselines $K$. Thus
\begin{equation}\label{eq:baseline}
\sum_PB_P=qK,\qquad
\sum_P(B_P/K)^2=q-p-3+4(p+1)\alpha^2=q-\gamma.
\end{equation}
The strict deficit $\gamma>0$ will yield word decay. All baselines are even; for example $\alpha K=2(p+3)H$.

\subsection{Slot counts and the full involution}
Put $R_P=\ind_{P\notin\{o_1,o_2\}}$ for $P\in L$. Direct counting gives
\begin{equation}\label{eq:slots}
\begin{array}{c|cc}
t&T_t&D_t\\\hline
P\in L&qB_P+u_P+R_P&d_P+R_P\\
x\notin L\cup\{z\},\ P=zx\cap L&qK+u_P&d_P\\
z&qK&0\\
e_i\ (i>0)&2q\alpha K+1&-r_0\\
e_0&2q\alpha K+2&1-r_0
\end{array}
\end{equation}
For extras, the normal total is $q(B_{o_1}+B_{o_2})+u_{o_1}+u_{o_2}=2q\alpha K+1$. Only $e_0$ also receives a root slot. The special pairs in \eqref{eq:special} satisfy \eqref{eq:balance}; for example
\[
\begin{array}{c|cc}
t&T_t&D_t\\\hline
e_0&2q\alpha K+2&1-r_0\\
Z_0&2q\alpha K+2&r_0-1\\
o_1&q\alpha K&1-r_0\\
v_1&q\alpha K&r_0-1\\
z&qK&0\\
w&qK&0
\end{array}
\]
The pairs $o_2,v_2$ and $e_i,Z_i$ for $i>0$ work in the same way.

For the remaining points of $L$, pair profiles $(d,u)$ with $(-d-2,u)$. Explicitly, for each even $j=2,4,\ldots,r_0-1$, the residual multiset contains $p-2$ copies each of $(-j,0)$ and $(j-2,0)$; pair them. For each odd $j=1,3,\ldots,r_0-4$ and each $\xi=\pm1$, pair $(j,\xi)$ with $(-j-2,\xi)$. These exhaust the remaining points and give equal totals $qK+u+1$ and opposite signed counts $d+1$.

For affine points other than $z$, pair directions with profiles $(d,u)$ and $(-d,u)$. Their profile classes have equal size, and each direction contains $q-1$ such affine points. Points in zero-profile directions can be paired among themselves since $q-1$ is even. Together with $z\leftrightarrow w$, these rules define a fixed-point-free involution on all of $T$ satisfying \eqref{eq:balance}.

\subsection{Sampling the graphs}
For large $H$, all multiplicities in \eqref{eq:multiplicity} are positive integers; the sufficient bound $2(p+3)H>r_0-1$ makes this explicit. With $N=q^2K/2$, their zero-sum corrections give
\begin{equation}\label{eq:sizes}
|A^+|=N+1,\qquad |A^-|=|B|=N.
\end{equation}
Connectivity is not assumed; the witnesses use the root components.

The $B$-slot counts are $qB_P/2$ at $P\in L$, $qK/2$ at affine labels, and $q\alpha K$ at extras. The same involution preserves them. Fix any $B$-vertex as $s$ before sampling. For one representative $t$ of each inverse pair, independently choose uniform bijections
\[
(A^+)_t\longrightarrow(A^-)_{\bar t},\qquad
(A^-)_t\longrightarrow(A^+)_{\bar t},\qquad
B_t\longrightarrow B_{\bar t}.
\]
Every slot is paired once, so the graphs are immersed and $\Gamma_A$ is bipartite. Girth conditioning below removes any loops or short cycles initially present.

\input{estimates.tex}
\input{noncollapse.tex}
\input{backmatter.tex}
\end{document}

%% file: estimates.tex
\section{Why long words have a strict decay estimate}
\subsection{Conditional turn frequencies}
The bounded corrections and the single root disappear from limiting frequencies as $H\to\infty$. The limiting normal distribution is the same in each $A$-color and in $B$.

For distinct outgoing labels $i,j$, let $w(i,j)$ be the conditional frequency that a vertex carrying $i$ also carries $j$. Set $w(i,i)=0$: a path arriving by $t$ uses the outgoing slot $\bar t$, and departure by $\bar t$ would reverse that edge. Define
\begin{equation}\label{eq:matrix}
\mathcal{M}_{t,u}=w(\bar t,u)^2.
\end{equation}
The square occurs when two path occurrences must carry the same word. Dependent and repeated edges require the image-graph argument in Section~\ref{sec:transfer}.

The limiting frequencies follow directly from the line counts:
\begin{itemize}
\item Given $P\in L$, its $q$ normal lines have equal weight. Every affine point lies on one of them, giving frequency $1/q$. Each extra has frequency $1$ for $P=o_1,o_2$ and $0$ otherwise.
\item Given an extra, there are $2q$ equally weighted anchor lines. An affine point lies on two, giving $1/q$; each anchor lies on $q$, giving $1/2$; every other extra is always present.
\item Given an affine point $x$, there is one normal line in each direction $P\in L$, of total baseline $qK$. Thus, for ordinary $y\ne x$,
\[
w(x,y)=\frac{B_{xy\cap L}}{qK},\qquad w(x,e_i)=\frac{2\alpha}{q}.
\]
\end{itemize}
Distinct points on $L$ have joint frequency zero, since their only common line is the omitted line $L$.

\subsection{A function controlling every row}
Some rows have large sums, so we compare $\mathcal{M}f$ with a positive function $f$. Put $\eta=\gamma/4$ and
\[
\mathcal H=\{Z_0,\ldots,Z_{p-1},v_1,v_2\},\qquad
f(t)=\begin{cases}
2p,&t\in\mathcal H,\\
1+\eta,&t\in(L\cup E\cup\{z\})\setminus\mathcal H,\\
1,&\text{otherwise}.
\end{cases}
\]
The heavy values occur at the inverse partners of extras and anchors, precisely the row indices whose conditioning labels cause large contributions.

For a row conditioned at an affine point, each direction supplies $q$ other ordinary points. The unweighted ordinary squared mass is
\[
\sum_{P\in L}q\left(\frac{B_P}{qK}\right)^2=1-\frac{\gamma}{q},
\]
by \eqref{eq:baseline}. Increasing the weights of light points on $L$ costs at most $\eta/q$; the target $z$ costs at most $4\eta/q^2$. The heavy targets, extras, and $z$ give an additional bound
\[
\frac{(2p-1)(4p+2)+4p(1+\eta)+4\eta}{q^2}
\le\frac{8p^2+5p-1}{q^2}\le\frac{12p^2}{q^2}.
\]
Here $\alpha<1$, $\eta\le1/4$, and $p\ge3$. Consequently
\begin{equation}\label{eq:affine-row}
(\mathcal{M}f)(t)\le1-\frac{3\gamma}{4q}+\frac{12p^2}{q^2}
\le1-\frac{\gamma}{2q}<1.
\end{equation}
The bound $q\ge100p^4$ suffices; for example $\gamma\ge1/(2p)$.

The other rows satisfy the following estimates:
\[
\begin{array}{c|c|c}
\text{conditioning label }\bar t&\text{upper bound for }(\mathcal{M}f)(t)&f(t)\\\hline
P\in L\setminus\{o_1,o_2\}&1+\eta/q^2&1+\eta\\
o_1\text{ or }o_2&1+p(1+\eta)+\eta/q^2&2p\\
e_i&p+\tfrac12+(p-\tfrac12)\eta+\eta/q^2&2p
\end{array}
\]
Each bound is strictly smaller than the last entry. For rows conditioned at affine points, $f(t)\ge1$, so \eqref{eq:affine-row} applies, including $t=w=\bar z$. Therefore
\begin{equation}\label{eq:contraction}
\mathcal{M}f\le\theta f\qquad\text{for some fixed }\theta<1.
\end{equation}
The strict gap may be small; positivity for fixed $p,q$ is sufficient.

\section{Transferring the probabilistic estimates}\label{sec:transfer}
The geometric argument uses systems of labelled paths. The unequal label pools and the exceptional root require the following adaptations of \cite[Section~3]{sep}.

\subsection{The bounded-path exclusion to be proved}
An \emph{occurrence} is an edge position on a path; different occurrences may traverse the same graph edge. Consider positive-length paths in $\Gamma_A\sqcup\Gamma_B$, all cyclically immersed and closed except possibly one immersed interval whose initial vertex is the predetermined vertex $s\in B$. An \emph{interval comparison} pairs two disjoint occurrence intervals of the same length. It requires either equal words in the given order, or inverse words read in reverse order. The domains of all comparisons are pairwise disjoint as occurrence intervals. A comparison is \emph{genuine} if every pair of corresponding occurrences traverses distinct underlying unoriented graph edges. An occurrence outside all comparison domains is \emph{unpaired}. Paths may otherwise share vertices and edges and may traverse an edge repeatedly.

\begin{proposition}[Bounded-path exclusion]\label{prop:bounded-path}
For sufficiently small fixed $c_0>0$, put $\ell_H=\lfloor c_0\log H\rfloor$. There is a constant $\varepsilon>0$, depending only on the fixed types and turn weights, with the following property. For every fixed positive integer $K_0$, real $C_0\ge1$, and nonnegative integer $I_0$, the conditional probability tends to zero that there exists a system as above with at most $K_0$ paths, total length $M$ satisfying
\[
\ell_H\le M\le C_0\ell_H,
\]
at most $I_0$ interval comparisons, all genuine, and at most $\varepsilon M$ unpaired occurrences. The probability is conditional on both graphs having girth at least $\ell_H$. The constant $\varepsilon$ is chosen before $K_0,C_0,I_0$, which are held fixed as $H\to\infty$. Systems without a rooted interval are included.
\end{proposition}

We first establish the weight and matching estimates used to prove this proposition.

\subsection{Normalizing unequal pools}
Write $H\kappa_i+O(1)$ for an $i$-pool size in one specified color or in $B$, and $H\lambda_{ij}+O(1)$ for the joint count of distinct outgoing labels $i,j$. The leading data are common to the three vertex sets and satisfy
\[
\kappa_i=\kappa_{\bar i}>0,\qquad\lambda_{ij}=\lambda_{ji}.
\]
Define
\begin{equation}\label{eq:normalized}
s(t,u)=\frac{\lambda_{\bar t,u}}{\sqrt{\kappa_t\kappa_u}}\quad(u\ne\bar t),
\qquad s(t,\bar t)=0.
\end{equation}
These weights are invariant under reversal: $s(t,u)=s(\bar u,\bar t)$. With $\mathcal{D}=\operatorname{diag}(\kappa_t)$, the matrix $(s(t,u)^2)$ is $\mathcal{D}\mathcal{M}\mathcal{D}^{-1}$, since $w(\bar t,u)=\lambda_{\bar t,u}/\kappa_t$. It has the same spectral radius as $\mathcal{M}$, so normalization preserves the strict contraction.

\subsection{Turns that exist only at the root}
Two distinct ordinary points on $L$ never occur together in a normal type, but may occur together at the exceptional $A$-root $r$. A zero limiting weight must not rule out such paths. On reduced turns put $\hats=s+\tau$ for a sufficiently small fixed $\tau>0$, retaining zero on cancelling turns. Continuity of the spectral radius preserves the strict squared-matrix bound. The new weights have a positive lower bound on reduced turns and, for large $H$, dominate normalized finite-$H$ counts, including the contributions from $r$. Write $\hats_{\max}=\max_{t,u\in T}\hats(t,u)$.

For $W=t_1\cdots t_h$, put $\hatP(W)=\prod_{i=1}^{h-1}\hats(t_i,t_{i+1})$. There are fixed $C>0$, $\vartheta<1$ such that
\begin{equation}\label{eq:word-decay}
\sum_{W\in T^h}\hatP(W)^2\le C\vartheta^h.
\end{equation}
Individual perturbed weights may exceed one; the squared-word sum still decays exponentially.

\subsection{Conditioning on large girth}
Let $\ell_H=\lfloor c_0\log H\rfloor$ and $d_{\max}=q+p+1$, with $c_0>0$ sufficiently small. In the initial matching, the expected number of edges on cycles shorter than $\ell_H$ is at most
\[
C\sum_{j<\ell_H}jC_{\mathrm{cyc}}^j=o(H).
\]
There are $O(H)$ vertices, every pool has size bounded below by a positive multiple of $H$, and a prescribed cycle imposes one matching constraint per edge.

A switching replaces a bad edge $xy$ and a distant nonbad edge $uv$ in the same oriented pool by $xv,uy$. It preserves all types and the $A$-bipartition. Distance greater than $2\ell_H$ and the absence of a short cycle through $uv$ prevent new short cycles. Only $o(H)$ choices are forbidden. Repeated switching therefore shows that the girth-conditioned sample space is nonempty.

The quantitative switching estimate of \cite[Lemma~2.2]{sep} also applies. After $m$ distinct edges have been prescribed feasibly, an additional specified edge in a pool of size $N_t=H\kappa_t+O(1)$ has conditional probability at most
\begin{equation}\label{eq:switching}
\frac1{N_t-m-r_H},\qquad r_H=O(d_{\max}^{2\ell_H+2})=o(H),
\end{equation}
when the denominator is positive: at least that many forward switches are available, with a uniquely determined inverse switch. For $O(\log H)$ distinct edge prescriptions $\mathcal E$, iteration gives
\begin{equation}\label{eq:prescriptions}
\mathbb P(\mathcal E\mid\text{girth}\ge\ell_H)
\le\exp(o(\log H))\prod_{e\in\mathcal E}(H\kappa_{t(e)})^{-1}.
\end{equation}
The individual pool constants must be retained. A coarse $C/H$ per edge would introduce an uncontrolled exponential factor in path length.

\subsection{Normalization in the image-graph count}
For a prescribed chain $t_1\cdots t_h$, internal vertices contribute joint-label counts and edges contribute reciprocal pool sizes. The leading constants telescope:
\begin{equation}\label{eq:telescope}
\frac{\prod_{i=1}^{h-1}\lambda_{\bar t_i,t_{i+1}}}{\prod_{i=1}^h\kappa_{t_i}}
=\frac{\prod_{i=1}^{h-1}s(t_i,t_{i+1})}{\sqrt{\kappa_{t_1}\kappa_{t_h}}}.
\end{equation}
Finite-$H$ counts are bounded with $\hats$. A bounded image-graph pattern has only boundedly many chain endpoints, so endpoint factors are bounded constants.

The source block argument retains all full blocks, including unmatched ones, and omits only $o(\log H)$ turns near chain ends and between blocks. Their weights cost at most
\[
\max\{1,\hats_{\max}\}^{o(\log H)}=\exp(o(\log H)),
\]
which accounts for weights exceeding one.

\subsection{Repeated edges and genuine comparisons}
\begin{proof}[Proof of Proposition~\ref{prop:bounded-path}]
Fix $K_0,C_0,I_0$ as in the proposition. Traversing the same graph edge twice does not itself create a rare coincidence; we therefore count the actual image graph before imposing genuine comparisons. All error estimates below are uniform over patterns with these fixed bounds.

Write $M$ for total path length, with $\ell_H\le M\le C_0\ell_H$. Let $\Delta$ be the actual image graph. Mark path endpoints, indexing starts, and vertices of image degree different from two. Girth and the total-length bound give only boundedly many chains between marks, as proved in \cite[Section~3]{sep}. If an image edge occurs $m_e$ times, immersion and girth give
\begin{equation}\label{eq:multiplicity-bound}
m_*:=\max_e m_e\le2(M/\ell_H+K_0)\le2(C_0+K_0).
\end{equation}
Indeed two occurrences of the same oriented edge on one immersed path are separated by at least $\ell_H$ positions. Each marked chain has constant multiplicity.

Level $j$ contains the edges with $m_e\ge j$ and their incident vertices. Let $V_j,E_j$ count those vertices and edges. Put $\iota=1$ if the rooted interval is present and $\iota=0$ otherwise, and set $k_1=\iota$, $k_j=0$ for $j>1$. Let $a(v)$ count the interval endpoints at $v$, with multiplicity. Each passage through $v$ couples an incoming traversal incidence to an outgoing incidence on a distinct half-edge, by immersion; cyclic immersion also couples the last and first incidences of a closed path. Only interval endpoints are left uncoupled. Thus
\[
2\max_{e\ni v}m_e
\le\sum_{\text{half-edges at }v}m_e+a(v).
\]
This local incidence coupling is separate from the interval comparisons. Summing over vertices, and using $\sum_va(v)=2\iota$, gives
\begin{equation}\label{eq:level-exponents}
\sum_jE_j=M,\qquad
\sum_jV_j=\sum_v\max_{e\ni v}m_e\le M+\iota,
\qquad
\sum_j(V_j-E_j-k_j)\le0.
\end{equation}
The exponent $k_1$ records the prescribed image $s$ of the interval's initial vertex; no rooted-vertex saving is claimed at later levels.

For each component of $\Delta$ assigned to $A$, choose a compatible bipartite coloring once. Every level inherits this coloring, even if it disconnects. There are at most $K_0$ image components, so these choices cost at most $2^{K_0}$. No new colors are chosen at later levels.

\paragraph{Blocks and weight bins.}
The common block decomposition of \cite[Section~3]{sep} subdivides the marked chains into full blocks of lengths $h_i=\sigma_H+O(1)$, where $\sigma_H\to\infty$ and $\sigma_H=o(\ell_H)$. The subdivision is chosen from the abstract chain and comparison pattern, and every traversal of the same chain uses the same blocks. The block length $\sigma_H$ may depend on this pattern; its asymptotic bounds and all error estimates are uniform over the permitted patterns for fixed $K_0,C_0,I_0$. Write $m_i$ for the traversal multiplicity of block $i$, and $W_i$ for its word in the chain's chosen orientation. Only $o(\ell_H)$ positions and turns are lost in chain-end fragments and block alignment.

Choose a fixed $\delta>0$ so small that \eqref{eq:word-decay} gives
\[
\sum_{W\in T^h}\hatP(W)^2\le e^{-2\delta h}
\]
for all sufficiently large $h$. Every block word is reduced and has positive weight. Define its bin
\[
d_i=\lfloor-\log\hatP(W_i)\rfloor.
\]
The positive lower bound on reduced-turn weights and the last squared-word estimate give a fixed $a_0>0$ such that, for large $H$,
\begin{equation}\label{eq:weight-bins}
\frac{\delta\sigma_H}{2}\le d_i\le a_0\sigma_H,
\qquad
\#\{W\in T^{h_i}:\ \hatP(W)>0,\ \lfloor-\log\hatP(W)\rfloor=d_i\}
\le e^{2d_i-\delta\sigma_H}.
\end{equation}
Indeed, each word in this bin has squared weight greater than $e^{-2d_i-2}$, while each individual word has weight at most $e^{-\delta h_i}$. These conclusions do not require individual turn weights to be at most one.

The aligned comparisons form a partial matching of block occurrences: each matched pair compares $\sigma_H-o(\sigma_H)$ consecutive letters, possibly with reversal and inversion. Forgetting occurrence indices gives a link multigraph on the underlying full blocks, including isolated blocks. For each component $\mathcal U$ put
\[
S_{\mathcal U}=\sum_{i\in\mathcal U}m_i,\qquad
z_{\mathcal U}=\max_{i\in\mathcal U}m_i,
\]
and let $b_{\mathcal U}$ count its unmatched block occurrences. If $\mathcal U$ has no self-link, the partners of a block of maximum multiplicity belong to other blocks, so
\begin{equation}\label{eq:block-balance}
2z_{\mathcal U}\le S_{\mathcal U}+b_{\mathcal U}.
\end{equation}
The alignment construction also gives
\begin{equation}\label{eq:block-totals}
\sigma_H\sum_{\mathcal U}S_{\mathcal U}=M+o(\ell_H),
\qquad
\sigma_H\sum_{\mathcal U}b_{\mathcal U}\le b+o(\ell_H),
\end{equation}
where $b$ is the number of unpaired edge occurrences. These are the block-alignment and link-counting statements of \cite[Section~3]{sep}; they involve only the occurrence comparisons and the abstract image graph.

A word on one block determines the word on the next along a link except for $o(\sigma_H)$ letters. A spanning tree of a link component therefore allows at most $\exp(o(\sigma_H)|\mathcal U|)$ extensions of a chosen root word. A self-link has only $\exp(o(\sigma_H))$ root words: a translation self-link has nonzero shift by genuineness and makes the word periodic apart from $o(\sigma_H)$ free letters, while a reflection self-link is incompatible with immersion and the fixed-point-free inverse involution. These are the same string-counting arguments as in the source. Reversal symmetry of $\hats$ permits choosing any chain orientations.

There are polynomially many abstract chain patterns in $\ell_H$. There are $O(\ell_H/\sigma_H)$ blocks and $O(\sigma_H)$ possible bins per block. Recording the bins consequently gives only $\exp(o(\ell_H))$ enlarged patterns.

\paragraph{The fixed-label stage estimate.}
Fix an enlarged pattern. A permitted level-$j$ label assignment $\omega$ is the restriction of a full word assignment satisfying the bins, comparisons, and immersion constraints. Let $Y_j(\omega)$ count injective realizations of the level-$j$ image graph with this label assignment and the inherited coloring. When the interval is present and $j=1$, its initial vertex must map to the fixed vertex $s$. We claim
\begin{equation}\label{eq:fixed-label-stage}
\mathbb E[Y_j(\omega)\mid\operatorname{girth}\ge\ell_H]
\le H^{V_j-E_j-k_j}
\exp\left(-\sum_{i:m_i\ge j}d_i+o(\ell_H)\right).
\end{equation}
To see the pool constants explicitly, an internal chain vertex with incoming letter $t$ and outgoing letter $u$ has at most $H\lambda_{\bar t,u}+O(1)$ possible images in its prescribed vertex set. Each other unfixed marked vertex has $O(H)$ choices. Prescribing the $E_j$ distinct edges costs the product in \eqref{eq:prescriptions}. Multiplying these vertex and edge bounds gives $H^{V_j-E_j-k_j}$ times the normalized turn products in \eqref{eq:telescope}, with bounded endpoint factors. The perturbed weights dominate the finite-$H$ joint counts, including exceptional-root turns. Retaining the full-block products and bounding the $o(\ell_H)$ omitted turns by $\max\{1,\hats_{\max}\}^{o(\ell_H)}$ gives
\[
\prod_{i:m_i\ge j}\hatP(W_i)
\le\exp\left(-\sum_{i:m_i\ge j}d_i\right),
\]
and proves \eqref{eq:fixed-label-stage}. Injectivity is needed for the distinct edge prescriptions; ignoring it when counting vertex choices only enlarges that count.

Let $X_j$ count all realizations with permitted level-$j$ labels. The number of permitted label assignments is at most
\begin{equation}\label{eq:stage-label-count}
\exp\left(o(\ell_H)+
\sum_{\substack{\mathcal U\text{ without a self-link}\\z_{\mathcal U}\ge j}}
\left(2\min_{i\in\mathcal U}d_i-\delta\sigma_H\right)\right).
\end{equation}
For a component with an active block, choose a root block attaining its minimum bin, count its words by \eqref{eq:weight-bins}, and extend along a spanning tree of links. A self-linked component contributes only its propagation error. Arbitrary words in chain-end fragments contribute $\exp(o(\ell_H))$. The count is a projection of feasible full \emph{word assignments}: the chosen root block may be inactive at level $j$, and no extension of a graph embedding is asserted.

Let $B_j$ be the product of the right-hand sides of \eqref{eq:fixed-label-stage} and \eqref{eq:stage-label-count}, so $\mathbb E[X_j\mid\operatorname{girth}\ge\ell_H]\le B_j$. On multiplying these numerical bounds over the levels, the total exponent of $H$ is nonpositive by \eqref{eq:level-exponents}. For a link component without a self-link, the remaining logarithmic contribution is
\[
-\sum_{i\in\mathcal U}m_i d_i
+2z_{\mathcal U}\min_{i\in\mathcal U}d_i
-\delta\sigma_Hz_{\mathcal U}
\le-\frac{\delta\sigma_H}{2}S_{\mathcal U}
+a_0\sigma_Hb_{\mathcal U}.
\]
For $2z_{\mathcal U}\le S_{\mathcal U}$ this follows from the lower bin bound; otherwise use \eqref{eq:block-balance} and the upper bin bound. A self-linked component has no positive label-count term, so the lower bin bound alone gives the same negative contribution. Summing and using \eqref{eq:block-totals} proves
\begin{equation}\label{eq:product-stage}
\log\prod_{j=1}^{m_*}B_j\le-\frac\delta2M+a_0b+o(\ell_H),
\end{equation}
where $b$ is the number of unpaired occurrences. The constants $\delta,a_0$ depend on the fixed types and weights, not on $K_0,C_0,I_0$. Choose once and for all
\[
\varepsilon=\min\left\{\frac14,\frac{\delta}{16(1+a_0)}\right\}.
\]
If $b\le\varepsilon M$, then for sufficiently large $H$ the product in \eqref{eq:product-stage} is at most $e^{-\delta M/4}$. Put $J_0=\max\{1,\lceil2(C_0+K_0)\rceil\}$. Any full realization forces $X_j\ge1$ at every level, while \eqref{eq:multiplicity-bound} bounds the number of levels by $J_0$. Therefore
\[
\mathbb P(\text{this enlarged pattern is realized}\mid\operatorname{girth}\ge\ell_H)
\le\min_j\mathbb E[X_j\mid\operatorname{girth}\ge\ell_H]
\le e^{-\delta M/(4J_0)}.
\]
The union bound over $\exp(o(\ell_H))$ enlarged patterns is at most
\[
\exp\bigl(o(\ell_H)-\delta\ell_H/(4J_0)\bigr)=o(1).
\]
We multiplied numerical bounds only to select a small one; no independence between levels is used. This proves the transferred bounded-pattern statement with $\varepsilon$ chosen before the pattern bounds.
\end{proof}

\subsection{Diameter and short closures}
For small-set expansion, use the coarser $C/H$ consequence of \eqref{eq:switching} for up to a small positive multiple of $H$ prescriptions, with cutoff $m\le\kappa_{\min}H/3$, where $\kappa_{\min}=\min_{t\in T}\kappa_t>0$. A set of $k$ vertices that fails to double after adjoining its neighbors lies in a set of $2k$ vertices containing at least $\lceil qk/2\rceil$ edges. Counting sets and prescriptions bounds the probability of failure at size $k$ by
\[
\bigl[C'(k/H)^{q/2-2}\bigr]^k,\qquad k\le c'H,
\]
for sufficiently small $c'>0$. Summation gives expansion with probability tending to one. Balls reach a fixed positive fraction of all vertices in $O(\log H)$ steps; packing disjoint balls along a geodesic gives diameter $O(\log H)$ for every component.

The short-closure argument of \cite[Lemma~4.2]{sep} applies: there are $O(H)$ vertices, logarithmic component diameter, and minimum degree at least three even after deleting the incident edges to be avoided at the ends of a segment. Every positive-length immersed segment can be completed to a cyclically immersed closed path by adding $O(\log H)$ edges, with no cancellation of the original segment.

%% file: noncollapse.tex
\section{The noncollapse needed for direct finiteness}\label{sec:noncollapse}
We now use the graph estimates solely to prove \eqref{eq:root}. The required deterministic input is planar localization from \cite[Section~4]{sep}; the reduction to a suitable disk picture is given below.

\subsection{Deterministic localization}
Proposition~\ref{prop:bounded-path} supplies one tolerance $\varepsilon>0$ before the bounds $K_0,C_0,I_0$ are chosen, and includes systems with no rooted interval. On the diameter event, there is a fixed $D$ such that every positive-length immersed segment has a cyclically immersed closure using at most $D\ell_H$ additional edges.

\begin{lemma}[Planar localization]\label{lem:planar-localization}
Fix $D$ and $\varepsilon>0$. There are fixed bounds $K_0,C_0,I_0$ with the following property for all sufficiently large $H$. Suppose $\Gamma_A\sqcup\Gamma_B$ has girth at least $\ell_H$, and every positive-length immersed segment extends without cancellation to a cyclically immersed closed path by appending at most $D\ell_H$ edges. Suppose also that a sphere contains finitely many pairwise disjoint disks satisfying the following conditions:
\begin{enumerate}
\item At least one ordinary boundary is present, and every ordinary boundary carries a positive-length cyclically immersed closed path in $\Gamma_A\sqcup\Gamma_B$.
\item Exactly one exceptional boundary carries an immersed interval from $s$ to a distinct vertex; its endpoints occur at one marked break, across which immersion is not required.
\item Every edge occurrence has one partner, joined to it by an arc in the complement of the disks. The arcs are mutually disjoint. Paired occurrences have inverse signed labels in the induced boundary orientations and traverse distinct underlying unoriented graph edges.
\end{enumerate}
Then the graphs admit a bounded path system as in Proposition~\ref{prop:bounded-path}, with these bounds and at most $\varepsilon M$ unpaired occurrences, where $M$ is its total length.
\end{lemma}
\begin{proof}
This is the deterministic extraction in the proof of \cite[Proposition~4.5]{sep}, with $\ell_H$ in place of its girth parameter $L$. Split long boundaries into segments, close those segments, apply planar separation to the pairing graph, and select one surviving cluster. Closure edges and occurrences whose partners were removed are left unpaired. The choices in that proof depend only on $D$ and $\varepsilon$ and give the asserted fixed bounds. Its extraction uses precisely the girth, closure, and planar-pairing hypotheses above.
\end{proof}

The localized system need not retain the exceptional boundary. A selected cluster may avoid it; a long exceptional interval may be split and all of its pieces closed. This is why entirely cyclic systems must remain in the probabilistic exclusion, even though only $s$ needs protection.

Choose the parameters in order: $p,q$ and the type data; word weights and their strict gap; a sufficiently small girth constant and the resulting closure constant $D$; the tolerance $\varepsilon$; the localization bounds; and finally large $H$. The diameter event and the exclusion event of Proposition~\ref{prop:bounded-path} hold simultaneously with probability $1-o(1)$ in the nonempty girth-conditioned space. Choose such a matching. Lemma~\ref{lem:planar-localization} shows that its graphs admit no arrangement of the stated kind. The localization uses graph paths and planar pairings, independently of coefficient characteristic and matching-pool sizes.

\subsection{A minimal disk picture protects the root}
\begin{lemma}[Root noncollapse]\label{lem:noncollapse}
For the chosen matching, no path in $\Gamma_B$ from $s$ to a vertex $y\ne s$ has trivial label in $G$.
\end{lemma}
\begin{proof}
Suppose such a path exists. Van Kampen's lemma supplies a disk picture: inner disks have boundary lifts to closed paths in the defining graphs, and their complement maps to the rose. The outer boundary reads a graph path from $s$ to $y$. We may allow arbitrary closed graph paths on inner boundaries, since their labels are relators in $G$.

Among all such pictures with fixed endpoints $s,y$, minimize the total length of all inner and outer boundaries, allowing the outer graph path to vary. Tighten inner boundaries cyclically and the outer path relative to its endpoints, and discard constant inner boundaries. Minimality makes every remaining inner path positive-length and cyclically immersed. The outer path is nonempty and immersed because $s\ne y$. At least one inner disk remains: otherwise the outer path would have a nonempty freely reduced label that is trivial in the free group of the rose.

Transverse preimages of interior points of the rose edges give disjoint arcs pairing all letter occurrences; closed preimage curves can be ignored. Cap the exterior by an exceptional disk, retaining the marked break between $y$ and $s$. With coherent boundary orientations the paired labels are inverse.

A pairing arc between occurrences of the same underlying graph edge can be thickened to a band and used to splice their lifts. Writing $e$ for an oriented edge, the four cases are:
\[
\begin{array}{l|l|l}
\text{boundaries involved}&\text{old paths}&\text{spliced paths}\\\hline
\text{two inner}&eP,\ e^{-1}Q&PQ\\
\text{outer and inner}&UeV,\ e^{-1}W&UWV\\
\text{one inner twice}&ePe^{-1}Q&P,\ Q\\
\text{outer twice}&UeVe^{-1}W&UW\text{ (outer)},\ V\text{ (other)}
\end{array}
\]
These are paths in the same defining graph wherever they are spliced, because the matched edge lifts coincide. Every new outer path has endpoints $s,y$.

For two inner disks, replace their disk-and-band union by the spliced inner disk. For an outer and an inner disk, absorb the latter into the exterior. A band returning to a single inner disk splits the picture after capping the two new closed graph boundaries; retain the part containing the outer boundary. For a band returning to the exterior, retain the complementary disk adjacent to the boundary containing the marked break. The other part can be discarded; see Figure~\ref{fig:splitting-surgeries}. Empty boundaries are discarded as well. These are the disk versions of the same-edge surgeries in \cite[Section~5]{sep}.

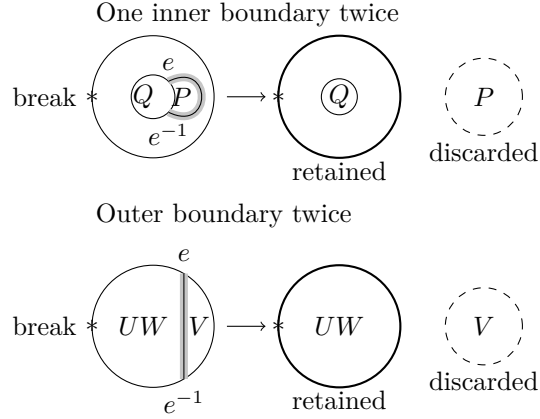
\begin{figure}[htbp]
\centering
\begin{tikzpicture}[x=0.85cm,y=0.85cm,every node/.style={font=\small}]
  \node[anchor=west] at (-1.05,1.28) {One inner boundary twice};
  \draw (0,0) circle (0.95);
  \draw (0,0) circle (0.34);
  \draw[line width=3pt,gray!45] (0.24,0.24) .. controls (0.93,0.58) and (0.93,-0.58) .. (0.24,-0.24);
  \draw (0.24,0.24) .. controls (0.93,0.58) and (0.93,-0.58) .. (0.24,-0.24);
  \node at (-0.95,0) {$\ast$};
  \node[anchor=east] at (-1.03,0) {break};
  \node at (0.44,0) {$P$};
  \node at (-0.15,0) {$Q$};
  \node[anchor=south] at (0.24,0.26) {$e$};
  \node[anchor=north] at (0.24,-0.26) {$e^{-1}$};
  \draw[->] (1.15,0) -- (1.75,0);
  \draw[thick] (2.9,0) circle (0.95);
  \draw (2.9,0) circle (0.28);
  \node at (2.9,0) {$Q$};
  \node at (1.95,0) {$\ast$};
  \node at (2.9,-1.15) {retained};
  \draw[dashed] (5.15,0) circle (0.60);
  \node at (5.15,0) {$P$};
  \node at (5.15,-0.85) {discarded};
  \node[anchor=west] at (-1.05,-1.85) {Outer boundary twice};
  \begin{scope}[yshift=-4mm]
  \draw (0,-3.1) circle (0.95);
  \draw[line width=3pt,gray!45] (0.48,-2.28) -- (0.48,-3.92);
  \draw (0.48,-2.28) -- (0.48,-3.92);
  \node at (-0.95,-3.1) {$\ast$};
  \node[anchor=east] at (-1.03,-3.1) {break};
  \node at (-0.15,-3.1) {$UW$};
  \node at (0.72,-3.1) {$V$};
  \node[anchor=south] at (0.48,-2.25) {$e$};
  \node[anchor=north] at (0.48,-3.95) {$e^{-1}$};
  \draw[->] (1.15,-3.1) -- (1.75,-3.1);
  \draw[thick] (2.9,-3.1) circle (0.95);
  \node at (2.9,-3.1) {$UW$};
  \node at (1.95,-3.1) {$\ast$};
  \node at (2.9,-4.25) {retained};
  \draw[dashed] (5.15,-3.1) circle (0.60);
  \node at (5.15,-3.1) {$V$};
  \node at (5.15,-3.95) {discarded};
  \end{scope}
\end{tikzpicture}
\caption{The two splitting surgeries, shown schematically with other inner disks suppressed. The gray band joins the paired occurrences of the same edge. After surgery the thick outer boundary belongs to the retained disk picture and carries the marked break $\ast$; the dashed component is discarded. The labels name the remaining boundary paths.}
\label{fig:splitting-surgeries}
\end{figure}

Each operation leaves a valid disk picture with the same outer endpoints and removes two boundary occurrences before any further tightening or discarding. In a split case, the sum of boundary lengths over both pieces has already decreased by two, so retaining only the outer-boundary piece also decreases it. This contradicts minimality.

All pairings are therefore genuine. The minimized picture gives precisely the arrangement excluded in the preceding subsection: at least one ordinary cyclic boundary, one $s$-rooted exceptional interval, and complete planar pairing over distinct graph edges. This contradiction proves the lemma.
\end{proof}

\subsection{Completion of the proof}
Lemma~\ref{lem:noncollapse} gives $h_y=1$ exactly when $y=s$. Hence \eqref{eq:fibre} yields $[c]_1=1$, so $c\ne0$. Lemma~\ref{lem:algebra} gives $ab=1$ and $ac=0$. Therefore $ba\ne1$, proving Theorem~\ref{thm:main}. For any field $k$ of characteristic $p$, the inclusion $\Fp[G]\hookrightarrow k[G]$ preserves these identities and the nonzero coefficient.

This proof establishes the noncollapse required by the algebraic certificate directly. It uses no essential-sphere argument, computation of $\pi_2$, contractibility of a universal cover, or torsion obstruction. The bounded-path and localization estimates remain the main geometric input.

%% file: backmatter.tex
\section*{Use of AI tools}
OpenAI's ChatGPT was used extensively for mathematical exploration, assistance with checking calculations and arguments, literature searches, and drafting and revising this manuscript. The named authors are responsible for the statements, proofs, citations, and final presentation. The AI-assisted checks are not a formal verification of the results. The earlier AI-generated manuscripts used as mathematical inputs are cited explicitly below.

\noindent\small The repository links in references~\cite{sep,oddone,oct} are pinned to commit \texttt{1581d4cbdd69} of October 12, 2026. Section, proposition, and equation references to these manuscripts refer to that revision.